\documentclass[12pt]{amsart}
\usepackage[a4paper,margin=3cm]{geometry}
\usepackage{amsmath,amssymb,mathtools}
\usepackage{microtype}
\usepackage[hidelinks]{hyperref}
\newcommand{\E}{\mathbb E}
\newcommand{\N}{\mathbb N}
\newcommand{\Z}{\mathbb Z}
\newcommand{\C}{\mathbb C}
\newcommand{\1}{\mathbf 1}
\newcommand{\e}{\mathrm e}
\newcommand{\calS}{\mathcal S}

\newtheorem{theorem}{Theorem}[section]
\newtheorem{proposition}[theorem]{Proposition}
\newtheorem{lemma}[theorem]{Lemma}

\title{Polynomial configurations and pointwise averages
along Piatetski--Shapiro sequences}
\author[H. MOUSAVI]{Hamed Mousavi}
    \address{Department of Mathematics, University of Bristol, Bristol, UK}
    \email {gj23799@bristol.ac.uk}
\date{}

\begin{document}

\begin{abstract}
In this paper, we prove that for every integer $k\geq2$ and every
$c>1$ sufficiently close to $1$, there is $\kappa>0$ such that every
sufficiently large subset of $\{1,\ldots,N\}$ of density at least
$(\log\log N)^{-\kappa}$ contains
\[
 x,\quad x+\lfloor n^c\rfloor,\quad x+\lfloor n^c\rfloor^2,
 \quad\ldots,\quad x+\lfloor n^c\rfloor^k.
\]
We also prove pointwise almost-everywhere convergence of the associated multiple ergodic averages.
\end{abstract}

\maketitle
\pagestyle{plain}

\section{Introduction}

Bergelson and Leibman proved the polynomial Szemer\'edi theorem
\cite{BergelsonLeibman}, which states that every set of integers of positive upper density contains polynomial configurations with a common parameter. Peluse subsequently obtained the first quantitative result for general distinct-degree families, proving that every sufficiently large subset of \([N]\) with density at least \((\log\log N)^{-\kappa}\) contains polynomial configurations \cite{Peluse}. Shao and Wang recently strengthened this to a polylogarithmic density bound, while also proving a quantitative popular-difference theorem \cite[Theorems~1.1 and~1.2]{ShaoWang}.

Sárközy established one of the first results where the common difference belongs to a sparse subset of integers with special structures. Indeed, he proved that every set of integers with positive density contains a two-term configuration with a square or shifted prime difference \cite{SarkozySquares,SarkozyShiftedPrimes}. In this direction, Piatetski--Shapiro sequences have been studied both
as restricted sets of common differences and as sparse sampling
sequences in ergodic theory.
Frantzikinakis and Wierdl proved that every set of positive upper density
contains arbitrarily long arithmetic progressions whose common difference
is of the form $\lfloor n^c\rfloor$
\cite[Theorem~A]{FrantzikinakisWierdl}.  For the configuration 
$x,x+\lfloor n^c\rfloor$, Rivat and S\'ark\"ozy obtained a power-saving
density bound \cite{RivatSarkozy}, which was recently strengthened for
$1<c<6/5$ by O'Keeffe \cite[Theorem~1.3]{OKeeffe}.
Boshernitzan and Wierdl proved pointwise convergence for single ergodic
averages along $\lfloor n^c\rfloor$ \cite{BoshernitzanWierdl}.  More
recently, Daskalakis established bilinear ergodic and Roth-type results
along fractional powers,  including configurations of the form $\{ x,x+\lfloor n^c\rfloor,
 x+2\lfloor n^c\rfloor\}$
for $c>1$ sufficiently close to $1$
\cite[Theorem~1.7]{Daskalakis}.  These results concern either a single
difference or equal-degree linear configurations.   In this paper, we establish a quantitative
distinct-degree polynomial Szemer\'edi theorem along
Piatetski--Shapiro integers and prove the corresponding pointwise
multiple ergodic theorem.

For $c>1$, write
\[
 \calS_c:=\{\lfloor n^c\rfloor:n\in\N\}.
\]
Fix throughout
\begin{equation}\label{e:polynomials}
 P_0=0,\qquad P_1,\ldots,P_\ell\in\Z[y],\qquad
 P_i(0)=0,\qquad
 1\leq\deg P_1<\cdots<\deg P_\ell=d.
\end{equation}
The vanishing constant-term assumption is standard in polynomial
recurrence.
Our first result is a quantitative nonlinear Szemer\'edi theorem\footnote{We do not recover the polylogarithmic threshold of
\cite{ShaoWang}: our comparison requires their popular-difference
estimate, whose quantitative dependence yields the double-logarithmic
bound above.} along
$\calS_c$.

\begin{theorem}[Piatetski--Shapiro polynomial Szemer\'edi theorem]
\label{t:szemeredi}
For every family satisfying~\eqref{e:polynomials}, there are $c_*>1$
and $\kappa>0$ such that the following holds.  If $1<c<c_*$, then every
sufficiently large $A\subset[N]$ satisfying
\[
 |A|\geq \frac{N}{(\log\log N)^\kappa}
\]
contains
\begin{equation}\label{e:pattern}
 x,\quad x+P_1(\lfloor n^c\rfloor),\quad\ldots,\quad
 x+P_\ell(\lfloor n^c\rfloor)
\end{equation}
for some positive integers $x,n$. Moreover, $x,n$ may be chosen so that the displayed terms are
pairwise distinct.
\end{theorem}
The argument also gives the following quantitative counting estimate,
based on the popular-difference theorem of Shao--Wang.  For
$A\subset[N]$ and $M=\lfloor N^{1/d}\rfloor$, define
\begin{align*}
 \mathcal C_{c,\mathbf P}(A;M):=
 \bigl|\bigl\{(x,n)\in\Z\times\N:{}&
 \lfloor n^c\rfloor\leq M,\ 
 x+P_i(\lfloor n^c\rfloor)\in A\quad(0\leq i\leq\ell)
 \bigr\}\bigr|.
\end{align*}
For every $1<c<c_*$, there are constants $C,\zeta>0$, depending only
on the fixed polynomial family, such that, writing
$0<\delta=|A|/N$,
\begin{equation}\label{e:quantitative-count}
 \frac{\mathcal C_{c,\mathbf P}(A;M)}{NM^{1/c}}
 \geq \exp\!\bigl(-\exp(C\delta^{-C})\bigr)
      -O_{c,\mathbf P}(M^{-\zeta})
\end{equation}
whenever $ N\geq\exp\!\bigl(\exp(C\delta^{-C})\bigr).$

The connection between Szemerédi-type theorems and multiple recurrence was studied by Furstenberg \cite{Furstenberg} and was extended to polynomial configurations by Bergelson and Leibman \cite{BergelsonLeibman}. This connection suggests a corresponding ergodic theorem analogous to Theorem~\ref{t:szemeredi}. For Piatetski--Shapiro times, Frantzikinakis
\cite[Theorem~2.1]{FrantzikinakisHardy} proved norm convergence for
linear multiple averages.  Daskalakis proved that single averages along orbits
$\lfloor n^c\rfloor^2$ converge pointwise to the ordinary quadratic
limit \cite[Theorem~1.13]{Daskalakis}.  In separate work, he proved
pointwise convergence for symmetric double recurrence averages along
$\calS_c$ when $1<c<23/22$
\cite[Theorem~1.7]{DaskalakisPointwise}.  Fornal and Krause subsequently
extended this range to
$1<c<7/6$, and also obtained a weighted version
\cite[Proposition~1.5 and Remark~4]{FornalKrause}. In this paper, we prove pointwise convergence for the general distinct-degree polynomial
orbits considered here.

 Let $(X,\nu,T)$ be an invertible probability measure-preserving system and
let $f_1,\ldots,f_\ell\in L^\infty(X)$.  Put
\begin{equation}\label{e:sparse-average}
 A_R^{(c)}(\mathbf f)(x):=
 \frac1R\sum_{n\leq R}\prod_{i=1}^{\ell}
 f_i\bigl(T^{P_i(\lfloor n^c\rfloor)}x\bigr).
\end{equation}
\begin{theorem}[Pointwise Piatetski--Shapiro averages]
\label{t:pointwise}
There exists $c_*>1$ such that the following statement  holds. Given $1<c<c_*$, the
averages~\eqref{e:sparse-average} converge almost everywhere, and their
limit agrees almost everywhere with the limit of
\begin{equation}\label{e:integer-average}
 \frac1M\sum_{m\leq M}\prod_{i=1}^{\ell}
 f_i\bigl(T^{P_i(m)}x\bigr).
\end{equation}
\end{theorem}

The usual PET induction bounds a weighted polynomial count by a 
Gowers norm. The inverse theorem then shows that a large polynomial count forces correlation
with a nilsequence.  In the case of averages along $\calS_c$, controlling
all such obstructions is unnecessary.  The appropriate input is the
following distinct-degree generalized von Neumann theorem of
Ter\"av\"ainen \cite[Theorem 4.1]{Teravainen}, which identifies ordinary polynomial phases as sufficient
obstructions. 

Write \([N]=\{1,\ldots,N\}\), \(\mathbf P=(P_1,\ldots,P_\ell)\), $\e(t):=e^{2\pi i t}$, $\sum_{m\le M} f(m) =\sum_{0<m\le M} f(m)$,
and \(\E_{m\leq M}=M^{-1}\sum_{m\leq M}\).
Unrestricted sums over \(x\) are over \(\Z\).
A function is \(1\)-bounded if its absolute value is at most \(1\).  For $w:[M]\to\C$, define
\begin{equation}\label{e:little-u}
 \|w\|_{u^s[M]}:=
 \sup_{\substack{\phi\in\mathbb R[y]\\ \deg\phi<s}}
 \left|\E_{m\leq M}w(m)\e(-\phi(m))\right|.
\end{equation}

\begin{theorem}[\cite{Teravainen}]
\label{t:teravainen}
Let $Q_1,\ldots,Q_\ell\in\Z[y]$ satisfy
$\deg Q_1<\cdots<\deg Q_\ell=d$, and let $C_0\geq1$.  There is an
integer $1\leq K\ll_d1$ such that the following holds.  If
$F_0,\ldots,F_\ell:\Z\to\C$ are $1$-bounded and supported on
$[-C_0M^d,C_0M^d]$, and $\theta:[M]\to\C$ is $1$-bounded, then
\begin{align}\label{e:generalized-vn}
 &\left|\frac1{M^{d+1}}\sum_{x\in\Z}\sum_{m\leq M}
 \theta(m)F_0(x)\prod_{i=1}^{\ell}F_i(x+Q_i(m))\right|\ll_{C_0,\mathbf Q}
 \bigl(M^{-1}+\|\theta\|_{u^{d+1}[M]}\bigr)^{1/K}.
\end{align}
\end{theorem}
We shall use the
following immediate scaled form.  If $|g|\leq B$ on $[M]$ and
\begin{equation}\label{e:phase-hypothesis}
 \sup_{\deg\phi\leq d}
 \left|\E_{m\leq M}g(m)\e(\phi(m))\right|\leq E,
\end{equation}
then, for $B>0$, the left-hand side of~\eqref{e:generalized-vn}, with $g$ in place
of $\theta$, is
\begin{equation}\label{e:scaled-vn}
 \ll_{C_0,\mathbf Q}B\bigl(M^{-1}+B^{-1}E\bigr)^{1/K}.
\end{equation}
Indeed, apply Theorem~\ref{t:teravainen} to $g/B$.

We use the estimate in \eqref{e:phase-hypothesis} to approximate the contribution from $\mathbf{1}_{\mathcal{S}_{c}}$ with a more tractable weighted counting form.
Vaaler's approximation reduces the problem to exponential sums with phase
\[
 Q(m)+hm^{1/c},\qquad \deg Q\leq d.
\]
Note that the $(d+1)$-st derivative removes $Q$. Combining this power-saving comparison with the quantitative counting
result of Shao and Wang
\cite[Proposition~5.1 and Lemma 5.2]{ShaoWang} proves
Theorem~\ref{t:szemeredi}.  Its orbit version, together with the
pointwise theorem of Kosz--Mirek--Peluse--Wan--Wright
\cite[Theorem~1.3]{KMPWW}, proves Theorem~\ref{t:pointwise}.

For the ergodic application we record the corresponding orbit estimate.

\begin{proposition}\label{p:orbit}
Let $(X,\nu,T)$ be an invertible probability measure-preserving system,
and let $Q_1,\ldots,Q_\ell$ be as in
Theorem~\ref{t:teravainen}.  Suppose that $B>0$, that $|g|\leq B$ on
$[M]$, and that $g$ satisfies~\eqref{e:phase-hypothesis}.  Then
\begin{align}\label{e:orbit-transfer}
 &\left\|\E_{m\leq M}g(m)
 \prod_{i=1}^{\ell}f_i(T^{Q_i(m)}\cdot)\right\|_{L^1(X)}\ll_{\mathbf Q}B
 \bigl(M^{-1}+B^{-1}E\bigr)^{1/K}
 \prod_{i=1}^{\ell}\|f_i\|_\infty.
\end{align}
\end{proposition}

\begin{proof}
By homogeneity assume $\|f_i\|_\infty\leq1$.  Since the
class of polynomial phases of degree at most $d$ is closed under
negation, \eqref{e:phase-hypothesis} gives
\[
 \|g\|_{u^{d+1}[M]}\leq E.
\]
Put $L=M^d$.  By $L^1$--$L^\infty$ duality, there is
$h\in L^\infty(X)$ with $\|h\|_\infty\leq1$ such that
\begin{align*}
 &\left\|\E_{m\leq M}g(m)
   \prod_{i=1}^{\ell}f_i(T^{Q_i(m)}\cdot)\right\|_{L^1(X)}=
 \left|\int_X h(x)\E_{m\leq M}g(m)
   \prod_{i=1}^{\ell}f_i(T^{Q_i(m)}x)\,d\nu(x)\right|.
\end{align*}
Since $\nu$ is $T$-invariant, averaging the integral over
$1\leq a\leq L$ gives
\begin{align*}
 &\left\|\E_{m\leq M}g(m)
   \prod_{i=1}^{\ell}f_i(T^{Q_i(m)}\cdot)\right\|_{L^1(X)}\\
 &\qquad=
 \left|\int_X\frac1L\sum_{a\leq L}h(T^ax)
   \E_{m\leq M}g(m)
   \prod_{i=1}^{\ell}f_i(T^{a+Q_i(m)}x)\,d\nu(x)\right|.
\end{align*}
For fixed $x$, define $F_0(a)=h(T^ax)\1_{[L]}(a)$ and define
$F_i(b)=f_i(T^bx)$ on an interval containing all $a+Q_i(m)$ with
$a\in[L]$ and $m\in[M]$, extending by zero elsewhere.  These intervals
may be chosen inside $[-C_0M^d,C_0M^d]$ for some
$C_0=C_0(\mathbf Q)$.  Thus the functions are $1$-bounded and have the
support required in Theorem~\ref{t:teravainen}.  Since $L=M^d$,
\eqref{e:scaled-vn} gives, uniformly in $x$,
\[
 \left|\frac1L\sum_{a\leq L}h(T^ax)\E_{m\leq M}g(m)
 \prod_{i=1}^{\ell}f_i(T^{a+Q_i(m)}x)\right|
 \ll_{C_0,\mathbf Q}B\bigl(M^{-1}+B^{-1}E\bigr)^{1/K}.
\]
Integrating this estimate over $X$ proves~\eqref{e:orbit-transfer}.
\end{proof}

\subsection*{Acknowledgments}
The author thanks Michael Lacey and Anastasios Fragkos for helpful
comments and discussions.  ChatGPT was used as an auxiliary tool for
proofreading and preliminary literature searches.  The author takes
full responsibility for all mathematical statements, references,
arguments, and final wording.

\section{A nonlinear Szemer\'edi theorem}

Put $\gamma=1/c$.  The map $n\mapsto\lfloor n^c\rfloor$ is strictly
increasing, since $(n+1)^c-n^c>1$.  Define
\begin{equation}\label{e:weights}
 \chi_c(m):=\1_{\calS_c}(m),\qquad
 \Delta_c(m):=(m+1)^\gamma-m^\gamma,\qquad
 r_c(m):=\chi_c(m)-\Delta_c(m).
\end{equation}
The exact counting identity
\[
 \chi_c(m)=\lceil(m+1)^\gamma\rceil-\lceil m^\gamma\rceil
\]
and $\lceil t\rceil=t+\psi(-t)+1/2$, where
$\psi(t)=\{t\}-1/2$, yield
\begin{equation}\label{e:sawtooth}
 r_c(m)=\psi(-(m+1)^\gamma)-\psi(-m^\gamma).
\end{equation}
We shall also use
\begin{equation}\label{e:delta-properties}
 0<\Delta_c(m)<1,\qquad |r_c(m)|\leq1,
 \qquad \sum_{m\leq M}\Delta_c(m)=(M+1)^\gamma-1.
\end{equation}

We require two standard exponential-sum tools. The first is the following consequence of Vaaler's approximation to the sawtooth function \cite[Theorem~18, especially~(7.14)]{Vaaler}. We include a short proof for completeness.

\begin{lemma}[Vaaler approximation]\label{l:vaaler}
For every integer $H\geq2$ there are coefficients $a_h,b_h$, with
$0<|h|\leq H$, and a nonnegative trigonometric polynomial
\[
 B_H(t)=b_0+\sum_{0<|h|\leq H}b_h\e(ht),
\]
such that
\[
 |a_h|\ll |h|^{-1},\qquad b_0\ll H^{-1},\qquad
 |b_h|\ll H^{-1},
\]
and
\[
 \left|\psi(t)-\sum_{0<|h|\leq H}a_h\e(ht)\right|
 \leq B_H(t)
\]
for every real $t$.
\end{lemma}

\begin{proof}
Vaaler uses the normalized sawtooth function
\[
 \widetilde\psi(t)=
 \begin{cases}
  \{t\}-\frac12,&t\notin\Z,\\
  0,&t\in\Z.
 \end{cases}
\]
Taking $N=H$ in \cite[Theorem~18]{Vaaler}, define
\[
 a_h:=(-2\pi i h)^{-1}\widehat J_{H+1}(h),
 \qquad 0<|h|\leq H.
\]
Then \cite[(7.14)]{Vaaler} gives
\[
 \left|\widetilde\psi(t)-\sum_{0<|h|\leq H}a_h\e(ht)\right|
 \leq \frac1{2H+2}k_H(t),
\]
where
\[
 k_H(t)=\sum_{|h|\leq H}
 \left(1-\frac{|h|}{H+1}\right)\e(ht)
\]
is the Fej\'er kernel.  Set $B_H=(2H+2)^{-1}k_H$.  The properties of
$\widehat J$ in \cite[Theorem~6]{Vaaler} give
$|\widehat J_{H+1}(h)|\leq1$, and hence
$|a_h|\ll|h|^{-1}$.  Moreover,
\[
 b_h=\frac1{2H+2}\left(1-\frac{|h|}{H+1}\right)
 \qquad (|h|\leq H),
\]
so $b_0,|b_h|\ll H^{-1}$. Finally, our convention $\psi(t)=\{t\}-1/2$ differs from
$\widetilde\psi$ only on $\Z$.  Since $\widehat J$ is even, we have
$a_{-h}=-a_h$.  Thus, if $t\in\Z$, then
\[
 \sum_{0<|h|\leq H}a_h\e(ht)
 =\sum_{h=1}^H(a_h+a_{-h})=0.
\]
Moreover, $k_H(t)=H+1$, and hence $B_H(t)=1/2$.  Since
$\psi(t)=-1/2$ at such points, the required inequality holds there as
well.
\end{proof}
The second is the following form of the classical higher derivative
test; see \cite[Lemma~2.2]{ChanKumchevWierdl}.

\begin{lemma}[Higher derivative test]\label{l:derivative-test}
Let $s\geq2$, let $J$ be an interval of $L$ integers, and suppose that
$F$ is $s$ times continuously differentiable on the convex hull of $J$.
If
\[
 \lambda\leq|F^{(s)}(t)|\leq C_1\lambda
\]
throughout that interval, then
\begin{equation}\label{e:derivative-test}
 \left|\sum_{m\in J}\e(F(m))\right|
 \ll_{s,C_1}L\left(\lambda^{1/(2^s-2)}+L^{-2/2^s}
 +(L^s\lambda)^{-2/2^s}\right).
\end{equation}
\end{lemma}

The required arithmetic input is the following uniform estimate for the
correlation of the Piatetski--Shapiro sampling discrepancy with
polynomial phases.

\begin{lemma}\label{l:fractional-phase}
Fix $d\geq1$ and $1<c_0<2$.  There are $\theta,\sigma_0>0$, depending
only on $d,c_0$, such that the following holds.  If $1<c<c_0$,
$\gamma=1/c$, $Q\in\mathbb R[y]$ has degree at most $d$, and
$1\leq|h|\leq M^\theta$, then
\begin{equation}\label{e:fractional-sum}
 \left|\sum_{m\leq M}\e(Q(m)+hm^\gamma)\right|
 \ll_{c,d}M^{1-\sigma_0}.
\end{equation}
The same estimate holds with $m^\gamma$ replaced by $(m+1)^\gamma$.
\end{lemma}

\begin{proof}
Choose $0<\theta<d/4$.  The range $m\leq M^{1/2}$ contributes
$O(M^{1/2})$.  Partition the remaining range into dyadic intervals
$J\subset[Y,2Y]$, where $Y\geq M^{1/2}$.

Since the $(d+1)$-st derivative of $Q$ vanishes, for $t\in[Y,2Y]$ we
have
\[
 \left|\frac{d^{d+1}}{dt^{d+1}}
 \bigl(Q(t)+ht^\gamma\bigr)\right|
 \asymp_{c,d}|h|Y^{\gamma-d-1}.
\]
Lemma~\ref{l:derivative-test} therefore gives
\begin{align*}
 \frac1{|J|}\left|\sum_{m\in J}\e(Q(m)+hm^\gamma)\right|
 \ll_{c,d}{}&
 |h|^{1/(2^{d+1}-2)}
 Y^{-(d+1-\gamma)/(2^{d+1}-2)}\\
 &+Y^{-2/2^{d+1}}+(|h|Y^\gamma)^{-2/2^{d+1}}.
\end{align*}
Since $Y\geq M^{1/2}$, $|h|\leq M^\theta$,
$d+1-\gamma\geq d$, and $\gamma\geq1/c_0$, these terms are bounded,
respectively, by
\[
 M^{-(d/2-\theta)/(2^{d+1}-2)},\qquad
 M^{-1/2^{d+1}},\qquad M^{-1/(c_0 2^{d+1})}.
\]
Summing over the dyadic intervals and combining this with the initial
range, and then choosing $\sigma_0>0$ smaller than the exponents
obtained above, proves~\eqref{e:fractional-sum}.

Finally, after the change of variables $r=m+1$, the sum with
$(m+1)^\gamma$ has polynomial part $Q(r-1)$ and differs from a sum of
the form~\eqref{e:fractional-sum} by at most two boundary terms.
\end{proof}

\begin{proposition}[Polynomial-phase cancellation]\label{p:phase}
Fix $d\geq1$ and $1<c_0<2$.  There is $\sigma>0$, depending only on
$d,c_0$, such that, for every fixed $1<c<c_0$ and all \(M\geq M_0(c)\),
\begin{equation}\label{e:phase-cancellation}
 \sup_{\deg Q\leq d}
 \left|\E_{m\leq M}r_c(m)\e(Q(m))\right|
 \ll_{c,d}M^{-\sigma},
\end{equation}
where \(r_c\) is defined in~\eqref{e:weights}.
\end{proposition}

\begin{proof}
Apply Lemma~\ref{l:vaaler} to the two sawtooth functions in
\eqref{e:sawtooth}, with $H=\lfloor M^\theta\rfloor$ and $\theta$ as
in Lemma~\ref{l:fractional-phase}.  We obtain
\begin{align*}
 \bigg|r_c(m)-\sum_{1\leq|h|\leq H}a_h\bigl(
 \e(-h(m+1)^\gamma)-\e(-hm^\gamma)\bigr)\bigg|
 \leq B_H(-(m+1)^\gamma)+B_H(-m^\gamma).
\end{align*}
Given $Q\in\mathbb R[y]$ with $\deg Q\leq d$, multiply by
$\e(Q(m))$ and sum over $m\leq M$.  It follows that
\begin{align}\label{e:rcmopen}
 \left|\sum_{m\leq M}r_c(m)\e(Q(m))\right|
 &\leq\bigg|\sum_{m\leq M}\e(Q(m))
 \sum_{1\leq|h|\leq H}a_h\bigl(
 \e(-h(m+1)^\gamma)-\e(-hm^\gamma)\bigr)\bigg|\nonumber\\
 &\quad+\sum_{m\leq M}\bigl(
 B_H(-(m+1)^\gamma)+B_H(-m^\gamma)\bigr).
\end{align}
By the triangle inequality and $|a_h|\ll|h|^{-1}$, the first term is
at most
\begin{align*}
 \ll\sum_{1\leq|h|\leq H}\frac1{|h|}\bigg(&
 \left|\sum_{m\leq M}\e(Q(m)-hm^\gamma)\right|
 +\left|\sum_{m\leq M}\e(Q(m)-h(m+1)^\gamma)\right|\bigg)\\
 &\ll_{c,d}M^{1-\sigma_0}
 \sum_{1\leq|h|\leq H}\frac1{|h|}
 \ll_{c,d}M^{1-\sigma_0}\log M,
\end{align*}
by Lemma~\ref{l:fractional-phase}.

For the second term in~\eqref{e:rcmopen}, the Fourier expansion of
$B_H$ and the triangle inequality give
\begin{align*}
 &\sum_{m\leq M}\bigl(
 B_H(-(m+1)^\gamma)+B_H(-m^\gamma)\bigr)\\
 &\qquad\leq 2Mb_0+
 \sum_{1\leq|h|\leq H}|b_h|
 \left(\left|\sum_{m\leq M}\e(-hm^\gamma)\right|
 +\left|\sum_{m\leq M}\e(-h(m+1)^\gamma)\right|\right)\\
 &\qquad\ll \frac{M}{H}+\frac1H
 \sum_{1\leq|h|\leq H}
 \left(\left|\sum_{m\leq M}\e(-hm^\gamma)\right|
 +\left|\sum_{m\leq M}\e(-h(m+1)^\gamma)\right|\right)\\
 &\qquad\ll_{c,d}M^{1-\theta}+M^{1-\sigma_0},
\end{align*}
where the last estimate follows from
Lemma~\ref{l:fractional-phase} with $Q=0$.  Dividing by $M$, taking the
supremum over $Q$, and choosing
$0<\sigma<\min\{\theta,\sigma_0\}$ proves the proposition.
\end{proof}

Combining Proposition~\ref{p:phase} with  Theorem~\ref{t:teravainen} gives the comparison used in both applications.

\begin{proposition}[Piatetski--Shapiro transference]\label{p:transfer}
There are $\beta>0$ and $c_*>1$, depending only on the family
\eqref{e:polynomials}, such that, whenever $1<c<c_*$ and
$F_0,\ldots,F_\ell$ are $1$-bounded and supported on a fixed multiple
of $[-M^d,M^d]$,
\begin{equation}\label{e:discrete-transfer}
 \left|\frac1{M^{d+1}}\sum_x\sum_{m\leq M}r_c(m)
 \prod_{i=0}^{\ell}F_i(x+P_i(m))\right|
 \ll_{c,\mathbf P}M^{-\beta}.
\end{equation}
The same constants may be chosen so that, writing $\gamma=1/c$,
\begin{equation}\label{e:c-choice}
 1-\gamma<\beta/2
\end{equation}
and
\begin{equation}\label{e:ergodic-transfer}
 \left\|\E_{m\leq M}r_c(m)
 \prod_{i=1}^{\ell}f_i(T^{P_i(m)}\cdot)\right\|_1
 \ll_{c,\mathbf P}M^{-\beta}\prod_i\|f_i\|_\infty.
\end{equation}
\end{proposition}

\begin{proof}
Use~\eqref{e:phase-cancellation} in
Theorem~\ref{t:teravainen} and Proposition~\ref{p:orbit}, with $B=1$.
Both right-hand sides are $O(M^{-\beta})$ for some $\beta>0$ depending
only on the fixed polynomial family and a fixed upper bound $c_0>1$.
The exponent is uniform for $1<c<c_0$; only the implied constant and
the threshold for $M$ may depend on $c$.  Finally shrink $c_*$ so that
$1-1/c<\beta/2$ for $1<c<c_*$.
\end{proof}

We recall the local-factor terminology from
\cite[Definitions~3.1, 3.2, 3.4 and~3.6]{ShaoWang}.  A factor
$\mathcal B$ of $[N]$ is a partition of $[N]$.  If $\mathcal B(x)$
denotes the atom containing $x$, define
\[
 (\Pi_{\mathcal B}f)(x)
 :=\frac1{|\mathcal B(x)|}\sum_{y\in\mathcal B(x)}f(y),
\]
extended by zero outside $[N]$.  The factor $\mathcal B$ is local of
modulus $q$ and resolution $M$ if every atom is an arithmetic
progression of step $q$ and cardinality between $M$ and $2M$.  A local
factor chain $\mathcal B_1,\ldots,\mathcal B_\ell$ of common modulus
$q$ is a sequence of local factors such that $\mathcal B_i$ refines
$\mathcal B_{i+1}$ for every $i<\ell$. 

The remaining combinatorial input is the following consequence of
Shao--Wang's quantitative popular-difference \cite[Theorem 1.2]{ShaoWang}. 

\begin{lemma}\label{l:shao-wang}
There is $C=C(\mathbf P)\geq1$ such that the following holds.  Let
$A\subset[N]$ have density $\delta$, and put
$M_0=\lfloor N^{1/d}\rfloor$.  If
\begin{equation}\label{e:shao-wang-size}
 N\geq\exp\!\bigl(\exp(C\delta^{-C})\bigr),
\end{equation}
then
\begin{equation}\label{e:shao-wang-count}
 \frac1{NM_0}\sum_{x\in\Z}\sum_{m\leq M_0}
 \prod_{i=0}^{\ell}\1_A(x+P_i(m))
 \geq\exp\!\bigl(-\exp(C\delta^{-C})\bigr).
\end{equation}
\end{lemma}

\begin{proof}
We spell out the deduction because Theorem~1.2 of \cite{ShaoWang}
produces one popular parameter, whereas we need a lower bound averaged
over all positive parameters.  Set
$\varepsilon=c_1\delta^{\ell+1}$, where
$c_1=c_1(\mathbf P)>0$ is sufficiently small.  After enlarging $C$,
the fixed polynomials have $(C,1)$-coefficients.  Since $\varepsilon=c_1\delta^{\ell+1}$, the size
condition~\eqref{e:shao-wang-size}, with $C$ chosen sufficiently large
in terms of $\mathbf P$, excludes the alternative $ N\leq\exp\!\bigl(\exp(O_{\mathbf P}
 (\varepsilon^{-O_d(1)}))\bigr)$
in \cite[Proposition~5.1]{ShaoWang}.\footnote{We use
Proposition~5.1 of \cite{ShaoWang} with the evident correction
\[
 M\geq N^{1/d}\exp\!\bigl(-\exp(O(\varepsilon^{-O(1)}))\bigr).
\]
The displayed lower bound in the published statement has misplaced
parentheses and is incompatible with the accompanying bound
\(M\leq N^{1/d}\).}
That proposition, applied with
$f_0=\cdots=f_\ell=\1_A$, therefore gives positive integers $q,M_1$
satisfying
\begin{equation}\label{e:shao-wang-parameters}
 M_0\geq M_1\geq
 M_0\exp\!\bigl(-\exp(C\delta^{-C})\bigr),
 \qquad q\leq\exp\!\bigl(\exp(C\delta^{-C})\bigr),
\end{equation}
and a local factor chain
$\mathcal B_1,\ldots,\mathcal B_\ell$ of modulus $q$ and resolutions
$M_1^{\deg P_1},\ldots,M_1^{\deg P_\ell}$ such that
\begin{align}\label{e:shao-wang-regularity}
 \bigg|&
 \frac1N\sum_x
 \E_{\substack{m\leq\varepsilon M_1\\q\mid m}}
 \prod_{i=0}^{\ell}\1_A(x+P_i(m))-\frac1N\sum_x
 \E_{\substack{m\leq\varepsilon M_1\\q\mid m}}
 \1_A(x)\prod_{i=1}^{\ell}
 (\Pi_{\mathcal B_i}\1_A)(x+P_i(m))
 \bigg|\leq\varepsilon.
\end{align}
Here $\Pi_{\mathcal B_i}f$ denotes the average of $f$ on each atom of
$\mathcal B_i$.
If $m\leq\varepsilon M_1$ and $q\mid m$, then
\[
 q\mid P_i(m),\qquad
 |P_i(m)|\ll_{\mathbf P}\varepsilon M_1^{\deg P_i},
\]
since $P_i(0)=0$.  Each atom of $\mathcal B_i$ is an arithmetic
progression of step $q$ and length comparable to
$M_1^{\deg P_i}$.  Thus the number of $x\in[N]$ for which $x$ and
$x+P_i(m)$ lie in different atoms is at most
\[
 \ll_{\mathbf P}
 \left(\frac{N}{M_1^{\deg P_i}}+1\right)
 \frac{|P_i(m)|}{q}
 \ll_{\mathbf P}\varepsilon N.
\]
Consequently,
\begin{align*}
 &\frac1N\sum_x
 \E_{\substack{m\leq\varepsilon M_1\\q\mid m}}
 \1_A(x)\prod_{i=1}^{\ell}
 (\Pi_{\mathcal B_i}\1_A)(x+P_i(m))=
 \frac1N\sum_x\1_A(x)\prod_{i=1}^{\ell}
 (\Pi_{\mathcal B_i}\1_A)(x)+O_{\mathbf P}(\varepsilon).
\end{align*}
Since $\mathcal B_1$ refines every $\mathcal B_i$, the product on the
right is $\mathcal B_1$-measurable.  Hence
\[
 \frac1N\sum_x\1_A(x)\prod_{i=1}^{\ell}
 (\Pi_{\mathcal B_i}\1_A)(x)
 =
 \frac1N\sum_x(\Pi_{\mathcal B_1}\1_A(x))^2
 \prod_{i=2}^{\ell}(\Pi_{\mathcal B_i}\1_A)(x).
\]
So \cite[Lemma~5.2]{ShaoWang}, applied to the chain
$\mathcal B_1,\mathcal B_1,\mathcal B_2,\ldots,\mathcal B_\ell$,
shows that this is at least $\delta^{\ell+1}$.  Combining this with
\eqref{e:shao-wang-regularity} and taking $c_1$ sufficiently small
gives
\begin{equation}\label{e:popular-average}
 \frac1N\sum_x
 \E_{\substack{m\leq\varepsilon M_1\\q\mid m}}
 \prod_{i=0}^{\ell}\1_A(x+P_i(m))
 \geq\delta^{\ell+1}-O_{\mathbf P}(\varepsilon)
 \geq\tfrac12\delta^{\ell+1}.
\end{equation}
The size assumption~\eqref{e:shao-wang-size}, with $C$ enlarged if
necessary, ensures $\varepsilon M_1/q\geq2$, and 
\[
 |\{m\leq\varepsilon M_1:q\mid m\}|
 =\left\lfloor\frac{\varepsilon M_1}{q}\right\rfloor
 \geq\frac{\varepsilon M_1}{2q}.
\]
Since $\varepsilon M_1\leq M_0$ and all summands are nonnegative,
opening~\eqref{e:popular-average} gives
\begin{align*}
 &\frac1{NM_0}\sum_x\sum_{m\leq M_0}
 \prod_{i=0}^{\ell}\1_A(x+P_i(m))\\
 &\geq
 \frac{|\{m\leq\varepsilon M_1:q\mid m\}|}{M_0}
 \left(
 \frac1N\sum_x
 \E_{\substack{m\leq\varepsilon M_1\\q\mid m}}
 \prod_{i=0}^{\ell}\1_A(x+P_i(m))
 \right)\\
 &\geq
 \frac{\varepsilon\delta^{\ell+1}}{4q}\frac{M_1}{M_0}.
\end{align*}
Using~\eqref{e:shao-wang-parameters} and
$\varepsilon=c_1\delta^{\ell+1}$, the last expression is at least
\[
 \exp\!\bigl(-\exp(C\delta^{-C})\bigr)
\]
after enlarging $C$.  This proves~\eqref{e:shao-wang-count}.
\end{proof}

\begin{proof}[Proof of Theorem~\ref{t:szemeredi}]
Put $M=\lfloor N^{1/d}\rfloor$, $\gamma=1/c$, and
\[
 F_A(x,m):=\prod_{i=0}^{\ell}\1_A(x+P_i(m)).
\]
Strict monotonicity of $n\mapsto\lfloor n^c\rfloor$ gives the exact
reindexing identity
\begin{equation}\label{e:count-reindex}
 \mathcal C_{c,\mathbf P}(A;M)
 =\sum_x\sum_{m\leq M}\chi_c(m)F_A(x,m).
\end{equation}
Since $N\asymp_d M^d$, Proposition~\ref{p:transfer} and
\eqref{e:weights} give
\begin{align}\label{e:count-decomposition}
 \frac{\mathcal C_{c,\mathbf P}(A;M)}{NM^\gamma}
 &=\frac1{NM^\gamma}\sum_{x,m\leq M}\Delta_c(m)F_A(x,m)
 +O_{c,\mathbf P}(M^{1-\gamma-\beta}).
\end{align}
The function $m\mapsto\Delta_c(m)$ is positive and decreasing, and the
mean value theorem gives
\begin{equation}\label{e:delta-lower}
 \Delta_c(M)\asymp_c M^{\gamma-1}.
\end{equation}
Consequently Lemma~\ref{l:shao-wang} implies
\begin{align*}
 \frac1{NM^\gamma}\sum_{x,m\leq M}\Delta_c(m)F_A(x,m)
 &\geq\frac{\Delta_c(M)}{M^{\gamma-1}}
 \left(\frac1{NM}\sum_{x,m\leq M}F_A(x,m)\right)\\
 &\gg_c\exp\!\bigl(-\exp(C\delta^{-C})\bigr).
\end{align*}
By~\eqref{e:c-choice}, the error in~\eqref{e:count-decomposition} is
$O(M^{-\beta/2})$.  Increasing $C$ to absorb the fixed positive
constant and taking $\zeta=\beta/2$ proves
\eqref{e:quantitative-count}.

Now assume $\delta\geq(\log\log N)^{-\kappa}$.  Choose $\kappa>0$ so
small that $\kappa C<1/2$.  Then
\[
 \exp(C\delta^{-C})
 \leq\exp\!\bigl(C(\log\log N)^{\kappa C}\bigr)=o(\log N),
\]
so the positive term in~\eqref{e:quantitative-count} is $N^{-o(1)}$.
It dominates the fixed power $M^{-\zeta}$ for all sufficiently large
$N$, and the count is positive.

Finally, for every \(1\leq i\leq\ell\), the polynomial \(P_i\) is nonzero,
and, for every \(1\leq i<j\leq\ell\), the polynomial \(P_i-P_j\)
is nonzero because \(P_i\) and \(P_j\) have distinct degrees.
Consequently, the equations
\[
 P_i(m)=0
 \quad\text{and}\quad
 P_i(m)=P_j(m)
\]
have only finitely many positive integer solutions.
 The contribution of all such $m$ to
\eqref{e:count-reindex} is $O_{\mathbf P}(N)$, but the preceding
lower bound is $NM^\gamma N^{-o(1)}$.  Removing these parameters leaves
a pairwise distinct configuration.
\end{proof}

\section{Pointwise ergodic averages}

The reference pointwise theorem  is
\cite[Theorem~1.3]{KMPWW}: for integer polynomials of distinct degrees,
the averages~\eqref{e:integer-average} converge pointwise almost
everywhere.  It remains to show that sampling the parameter along
$\calS_c$ does not change their limit.

We first record the elementary summability fact needed for the smooth
model.

\begin{lemma}[Riesz means]\label{l:riesz}
Let $(z_m)$ be a bounded complex sequence and suppose
\[
 \frac1M\sum_{m\leq M}z_m\longrightarrow z.
\]
For $0<\gamma<1$, put
$\Delta(m)=(m+1)^\gamma-m^\gamma$ and
$D_M=\sum_{m\leq M}\Delta(m)$.  Then
\begin{equation}\label{e:riesz-limit}
 \frac1{D_M}\sum_{m\leq M}\Delta(m)z_m\longrightarrow z.
\end{equation}
\end{lemma}

\begin{proof}
Put $ a_k:=\frac1k\sum_{m\leq k}z_m,$
so that \(a_k\to z\).  Since \(0<\gamma<1\), the sequence
\(\Delta(m)\) is positive and decreasing.  Summation by parts gives
\[
 \frac1{D_M}\sum_{m\leq M}\Delta(m)z_m
 =
 \sum_{k<M}\frac{k(\Delta(k)-\Delta(k+1))}{D_M}a_k
 +\frac{M\Delta(M)}{D_M}a_M.
\]
The coefficients are nonnegative, tend to zero for each fixed \(k\),
and, by telescoping, have sum
\[
 \frac{\sum_{k<M}k(\Delta(k)-\Delta(k+1))+M\Delta(M)}
 {D_M}=1.
\]
Thus they satisfy the hypotheses of the Toeplitz theorem
\cite[Theorem~2.6]{DaskalakisPointwise}, which proves
\eqref{e:riesz-limit}.
\end{proof}

\begin{proof}[Proof of Theorem~\ref{t:pointwise}]
By homogeneity suppose $\|f_i\|_\infty\leq1$, and write
\[
 z_m(x):=\prod_{i=1}^{\ell}f_i(T^{P_i(m)}x),
 \qquad M_R:=\lfloor R^c\rfloor,\qquad \gamma=1/c.
\]
Reindexing by $m=\lfloor n^c\rfloor$ gives
\begin{equation}\label{e:ergodic-reindex}
 A_R^{(c)}(\mathbf f)=\frac1R\sum_{m\leq M_R}\chi_c(m)z_m.
\end{equation}
Using~\eqref{e:weights},~\eqref{e:c-choice}, and~\eqref{e:ergodic-transfer},
\begin{align}
 &\left\|A_R^{(c)}(\mathbf f)-\frac1R
 \sum_{m\leq M_R}\Delta_c(m)z_m\right\|_1\ll_{c,\mathbf P}\frac{M_R}{R}M_R^{-\beta}
 \ll_{c,\mathbf P}M_R^{1-\gamma-\beta}
 \ll_{c,\mathbf P}R^{-c\beta/2} .                   \label{e:l1-power}
\end{align}

By \cite[Theorem~1.3]{KMPWW}, for almost every $x$ the averages
$M^{-1}\sum_{m\leq M}z_m(x)$ converge to a limit $L(x)$.
Lemma~\ref{l:riesz}, applied pointwise, gives
\begin{equation}\label{e:smooth-limit}
 \frac1{(M+1)^\gamma-1}\sum_{m\leq M}\Delta_c(m)z_m(x)
 \longrightarrow L(x).
\end{equation}
Moreover,
\begin{equation}\label{e:normalization}
 (M_R+1)^\gamma-1=R+O(1).
\end{equation}
Combining~\eqref{e:smooth-limit}, with $M=M_R$, and
\eqref{e:normalization}, we obtain
\begin{align*}
 \frac1R\sum_{m\leq M_R}\Delta_c(m)z_m(x)
 &=
 \frac{(M_R+1)^\gamma-1}{R}
 \left(
 \frac1{(M_R+1)^\gamma-1}
 \sum_{m\leq M_R}\Delta_c(m)z_m(x)
 \right)\longrightarrow L(x)
\end{align*}
for almost every $x$.

Choose $A>1+ \frac{2}{c\beta}$ and put $R_j=\lfloor j^A\rfloor$.  By
\eqref{e:l1-power},
\begin{align*}
 \sum_{j=1}^{\infty}
 \left\|A_{R_j}^{(c)}(\mathbf f)-\frac1{R_j}
 \sum_{m\leq M_{R_j}}\Delta_c(m)z_m\right\|_1
 &\ll_{c,\mathbf P}\sum_{j=1}^{\infty}R_j^{-c\beta/2}\ll_A\sum_{j=1}^{\infty}j^{-Ac\beta/2}<\infty.
\end{align*}
Consequently, by Tonelli's theorem,
\begin{align*}
 \int_X\sum_{j=1}^{\infty}
 \left|A_{R_j}^{(c)}(\mathbf f)(x)-\frac1{R_j}
 \sum_{m\leq M_{R_j}}\Delta_c(m)z_m(x)\right|\,d\nu(x)<\infty.
\end{align*}
Thus, for almost every $x$,
\[
 \left|A_{R_j}^{(c)}(\mathbf f)(x)-\frac1{R_j}
 \sum_{m\leq M_{R_j}}\Delta_c(m)z_m(x)\right|\longrightarrow0.
\]
Since the second term converges almost everywhere to $L(x)$, it follows
that
\[
 A_{R_j}^{(c)}(\mathbf f)(x)\longrightarrow L(x)
\]
for almost every $x$.

It remains to pass from $(R_j)$ to all positive integers.  If
$R_j\leq R<R_{j+1}$, then
\begin{align*}
 A_R^{(c)}(\mathbf f)(x)-A_{R_j}^{(c)}(\mathbf f)(x)
 &=
 \left(\frac1R-\frac1{R_j}\right)
 \sum_{n\leq R_j}\prod_{i=1}^{\ell}
 f_i\bigl(T^{P_i(\lfloor n^c\rfloor)}x\bigr)\\
 &\quad+\frac1R\sum_{R_j<n\leq R}\prod_{i=1}^{\ell}
 f_i\bigl(T^{P_i(\lfloor n^c\rfloor)}x\bigr).
\end{align*}
Since $\|f_i\|_\infty\leq1$, taking absolute values gives
\begin{align*}
 \left|A_R^{(c)}(\mathbf f)(x)-A_{R_j}^{(c)}(\mathbf f)(x)\right|
 &\leq
 \left(\frac1{R_j}-\frac1R\right)R_j+\frac{R-R_j}{R}=2\frac{R-R_j}{R}
 \leq2\frac{R_{j+1}-R_j}{R_j}.
\end{align*}
Moreover, the mean value theorem gives
\[
 R_{j+1}-R_j
 \leq (j+1)^A-j^A+1\ll_A j^{A-1},
 \qquad R_j\gg_A j^A.
\]
Hence
\[
 \sup_{R_j\leq R<R_{j+1}}
 \left|A_R^{(c)}(\mathbf f)(x)-A_{R_j}^{(c)}(\mathbf f)(x)\right|
 \ll_A j^{-1}\longrightarrow0
\]
uniformly in $x$.  This fills the gaps and proves convergence along the
full sequence, with the same limit as~\eqref{e:integer-average}.
\end{proof}

\end{document}